\documentclass[12pt]{amsart}

\makeatletter
\def\@settitle{\begin{center}%
		\baselineskip14\p@\relax
		\normalfont\Large
		\@title
	\end{center}%
}
\makeatletter
\def\@settitle{\begin{center}%
		\baselineskip14\p@\relax
		\normalfont\Large
		\@title
	\end{center}%
}
\makeatother

\usepackage{graphicx,epsfig}

\usepackage{bbm}
\usepackage[title]{appendix}
\usepackage{esvect}
\usepackage{mathtools}
\usepackage{tikz-cd}
\usepackage{fancyhdr}
\usepackage{amsfonts}
\usepackage{amsmath}
\usepackage{amssymb}
\usepackage{mathabx}
\usepackage{amsthm}
\usepackage{tikz-cd}
\usepackage{parskip}
\usepackage{times} 

\usepackage{mathrsfs}
\usepackage{scalerel,amssymb}

\usepackage{amsmath,amssymb,latexsym, amsfonts, amscd, amsthm, xy}
\input{xy}
\xyoption{all}

\newcommand{\N}{\mathbb{N}}
\newcommand{\Z}{\mathbb{Z}}
\newcommand{\Q}{\mathbb{Q}}
\newcommand{\R}{\mathbb{R}}
\newcommand{\C}{\mathbb{C}}

\newcommand{\SL}{\mathrm{SL}}

\newcommand{\GL}{\mathrm{GL}}

\input{xy}
\xyoption{all}

\makeindex 

=650pt \headheight = 8pt \headsep = 10pt

\usepackage{fancyhdr}
\usepackage{amsfonts}
\usepackage{amsmath}
\usepackage{amssymb}
\usepackage{amsthm}
\usepackage{tikz}
\usepackage{tikz-cd}
\usepackage{parskip}
\usepackage{times} 

\usepackage[
bibstyle=alphabetic,
citestyle=alphabetic,
hyperref=true,
backref=false]{biblatex}
\usepackage{hyperref}
\hypersetup{pdftoolbar=true, pdftitle={GKform}, pdffitwindow=true, colorlinks=true, citecolor=green, filecolor=black, linkcolor=blue, urlcolor=blue, hypertexnames=false}
\usepackage{cleveref}
\makeindex 

=650pt \headheight = 8pt \headsep = 10pt

\newtheorem{thm}{Theorem}[section]
\newtheorem{prop}[thm]{Proposition}
\newtheorem{lem}[thm]{Lemma}

\theoremstyle{definition}
\newtheorem{rem}[thm]{Remark}

\numberwithin{equation}{section}

\title[Eisenstein series in equivariant cohomology]{Classical periods of Eisenstein series and Bernoulli polynomials in the equivariant cohomology of a torus}
\author{Peter Xu}

\usepackage{setspace}
\begin{document}
	
	\makeatletter
	\@setabstract
	\makeatother
	\maketitle
	
	\begin{abstract}
		We find group cochains valued in currents giving explicit representatives for the $\GL_2$-equivariant polylogarithm class of a torus. Based on the construction of weight-$2$ Eisenstein series for $\GL_2$ from this polylogarithm class, we give a geometrically-flavored derivation of the classical formulas for the associated Dedekind-Rademacher homomorphisms, i.e. the periods of $E^2_{\alpha,\beta}$ for various nonzero torsion sections $(\alpha,\beta)$.
	\end{abstract}
	
	\tableofcontents
	
	\section{Introduction: classical Eisenstein series}
	
	Consider the classical weight-$2$ Eisenstein series by the analytic continuation/Hecke regularized series
	\begin{equation} \label{eq:eis}
		E_2^{\alpha,\beta}(\tau):= \left(\sum_{m,n\in \mathbb{Z}^2 \setminus 0} \frac{\exp(m\alpha+n\beta)}{|m\tau + n|^{2s} (m\tau +n)^2}\right)_{s=0}
	\end{equation}
	where $\tau\in\mathcal{H}$ is a variable in the upper half-plane and $(\alpha,\beta) \in (\frac{1}{N}\mathbb{Z}/\mathbb{Z})^2-\{(0,0)\}$ parameterizes an additive character of $\mathbb{Z}^2$ of conductor dividing $N$, for some integer $N> 1$. We may assume $N$ is minimal, i.e. $(\alpha,\beta)$ is a primitive vector, so that the conductor is precisely $N$. Then $E_2^{\alpha,\beta}(\tau)$ is a holomorphic modular form of level $\Gamma(N)$; see, e.g., \cite[III(11)]{Weil}.
	
	A classical formula due to Siegel \cite{Sieg} gives formulas for the periods of these modular forms in terms of \emph{periodic Bernoulli polynomials}: we write 
	\begin{equation} \label{eq:bers}
		B_2(x):=\{x\}^2-\{x\}+1/6, B_1(x):=
		\begin{dcases*}
			0
			& if  $x\in \Z$\,, \\[1ex]
			\{x\}-1/2 & otherwise\,.
		\end{dcases*}
	\end{equation}
	for the first two periodic Bernoulli polynomials, thought of as functions on $\mathbb{R}/\mathbb{Z}$. Here, $\{x\}$ denotes the ``fractional part'' function which takes the remainder modulo $1$ (i.e. in $[0,1)$) of any real number. Then associated to any matrix 
	\[
	\gamma := \begin{pmatrix} a&b \\ c&d \end{pmatrix} \in \Gamma(N),
	\]
	the integral
	\[
	\Phi_{\alpha,\beta}(\gamma):=\int_{\tau_0}^{\gamma \tau_0} E_2^{\alpha,\beta}(\tau) d\tau
	\]
	is independent of the basepoint $\tau_0$, and yields a homomorphism $\Gamma(N)\to \mathbb{Q}$ given by 
	\begin{equation} \label{eq:classical}
		\Phi_{\alpha,\beta}(\gamma)= 
		\begin{dcases*}
			\frac{b}{d} B_2(\beta)
			& if  $c=0$\,, \\[1ex]
			\frac{a+d}{c} B_2(\beta)-2\sum_{i=0}^{|c|-1} B_1\left(\frac{\beta+i}{|c|}\right)B_1\left(a\frac{\beta+i}{c}-\alpha\right)
			& otherwise\,.
		\end{dcases*}
	\end{equation}
	This formula exhibits the \emph{a priori} non-obvious fact that the various Eisenstein series $E^{\alpha,\beta}_2$, which together span all the holomorphic Eisenstein classes in the first cohomology of any arithmetic subgroup of the modular group, give \emph{rational} classes. In fact, by taking suitably smoothed combinations (``stabilizations''), one can even deduce integrality results.
	
	The original proof of this formula, in \cite[Theorem 13]{Sieg}, was by analytic means, grouping terms in infinite series and dealing with delicate convergence issues. Our proof is almost entirely algebraic, and closely tied to the underlying linear algebraic geometry of the $2$-torus with $\GL_2$-action.
	
	The content of the article can be viewed as an algebraic approach to the construction of the Bernoulli-Eisenstein classes of Sczech \cite{Scz1}. 
	
	\subsection{Acknowledgments}
	
	I would like to thank Marti Roset Julia for helping me to understand how to reconcile the explicit formulas, and Nicolas Bergeron for his inspiration in pursuing this work.
	
	\section{Equivariant/de Rham Eisenstein classes}
	
	We follow the treatment of \cite{BCG} to identify $E_2^{\alpha,\beta}$ with a class in cohomology. Let $\Gamma$ be a torsion-free subgroup of $\text{GL}_2(\mathbb{Z})$;\footnote{The torsion-free assumption is not strictly necessary, but we make it for simplicity.} then the $\Gamma$-action on the upper half-plane $\mathcal{H}$ by Möbius transformations 
	\[
	\begin{pmatrix}
		a&b\\c&d
	\end{pmatrix} \tau := \frac{a\tau +b}{c\tau+d}
	\]
	is free and therefore $Y(\Gamma):= \Gamma \backslash \mathcal{H}$ is a model for the classifying space of $\Gamma$. The torus bundle over the classifying space
	\[
	{T}(\Gamma) := \Gamma \backslash (\mathcal{H}\times (\C/\mathbb{Z}^2))
	\]
	then corresponds to the torus $T:=\mathbb{C}/\mathbb{Z}^2$ with left $\Gamma$-action given by the monodromy action on fibers (i.e. by the left standard representation of $\text{GL}_2(\mathbb{Z})$), in the sense that there is a natural isomorphism 
	\begin{equation}\label{eq:dict}
		H^\bullet_\Gamma (T)\cong H^\bullet({T}(\Gamma))
	\end{equation}
	between the equivariant cohomology of $T$ and the cohomology of the total space of the bundle $\mathbf{T}(\Gamma)$.
	
	Here, ``natural'' means that this isomorphism exists for not just the torus, but for \emph{any} $\Gamma$-space and corresponding bundle over $Y(\Gamma)$, and is functorial for morphisms between such objects. See, for example, the second author's thesis \cite[\S3]{X} for an explanation of this dictionary between bundles over the classifying space and equivariant spaces.
	
	Then following \cite[\S3]{BCG}, if we fix an auxiliary integer $c>1$, there is a unique class
	\[
	z_\Gamma^{(c)} \in H^{1}_\Gamma(T-T[c], \mathbb{Z}[1/c])
	\]
	characterized by being invariant by $[a]_*$ (for all integers $a$ relatively prime to $c$, which we denote by $\N^{(c)}$) and having residue 
	\[
	[T[c]-c^2\{0\}]\in H^2_{\Gamma,T[c]}(T, \mathbb{Z}[1/c])\cong H^0(T[c])^\Gamma,
	\]
	where $H^\bullet_{\Gamma,T[c]}(T)$ denotes the equivariant cohomology with support. When $\Gamma=\Gamma(N)$ and a $\Gamma$-fixed point $(\alpha,\beta)\in \frac{1}{N}\Z^2/\Z = T[N]$, the formulas of \cite[\S9]{BCG} tell us that that for any $c>1$ with $(c\alpha, c\beta)=(\alpha,\beta)$.
	\[
	(\alpha,\beta)^* \frac{1}{1-c^2}z_\Gamma^{(c)} \in H^1_\Gamma(*, \Z[1/c])=H^1(\Gamma, \Z[1/c])
	\]
	has a de Rham representative $\frac{1}{2}[E_2^{\alpha,\beta}]\in H^1(Y(\Gamma), \Z[1/c])$ as defined in the introduction, under the identification \eqref{eq:dict}. 
	
	\subsection{Computing with complexes}
	We now define the main tool we need, in this setting: the \emph{distributional de Rham complex} of $T$. Write $\mathcal{D}_T^i$ for the real-valued smooth $i$-currents on $T$, i.e. the linear dual of the compactly-supported smooth $(2-i)$-forms $\Omega^{2-i}_{T,c}$. The exterior derivative $d: \mathcal{D}^{i}_T\to  \mathcal{D}^{i+1}_T$ is defined as the graded adjoint of the exterior derivative on forms, i.e.
	\[
	(dc)(\eta) := (-1)^{\deg c} c(d\eta).
	\]
	With this differential, the currents form a complex
	\[
	\mathcal{D}^{0}_T\to \mathcal{D}^{1}_T\to \mathcal{D}^{2}_T,
	\]
	computing the real cohomology of $T$. There is an injective quasi-isomorphism from the usual de Rham complex
	\begin{equation} \label{eq:formcurrent}
		\upsilon: \Omega^{i}_{T} \hookrightarrow \mathcal{D}^{i}_T, \hookrightarrow \omega \mapsto \left(\eta \mapsto \int_{T}\eta \wedge \omega\right)
	\end{equation}
	whose injectivity can be seen locally by the non-degeneracy of the wedge product pairing. Via this map, we can and will implicitly view smooth forms as currents. 
	
	We would like to say that the map $\upsilon$ is a natural isomorphism, but this functoriality actually fails, because the integral over $T$ depends on its orientation, and so is reversed in sign by orientation-reversing maps. Consequently, for orientation-reversing maps, the action on currents can fail to give the correct action on cohomology. For example, on $S^1$, pushforward by the inverse map $[-1]_*$ sends the $1$-form $dz\mapsto -dz$, as it does for the associated cohomology class, but if we take the ``natural'' adjoint action, we have:
	\[
	[-1]_*(\upsilon(dz))(\eta) =(\upsilon(dz))([-1]^*\eta) = \int_{S^1} [-1]^*\eta \wedge dz = \int_{S^1}\eta(-z)\, dz = \int_{S^1}\eta(z)\, dz = (\upsilon(dz))(\eta)
	\]
	for a compactly supported smooth $0$-form (i.e. function on $S^1$) $\eta$, so we see that $[-1]_*$ actually fixes the associated current.\footnote{Philosophically, what is happening is that the distributional de Rham complex is really computing Borel-Moore homology, not cohomology. The usual isomorphism between the two is via Poincaré duality, which depends on a choice of orientation class.} 
	
	This subtlety addressed, we now introduce an important class of currents: associated to closed oriented submanifolds $Z\subset T$ of codimension $s$, we have a closed \emph{current of integration}
	\[
	\delta_Z \in \mathcal{D}^s_T
	\]
	defined by  
	\[
	\delta_Z(\omega):=\int_Z \omega.
	\]
	In line with our discussion above, reversing the orientation of $Z$ turns $\delta_Z$ into $-\delta_{-Z}$. Further, a current $\omega\in \mathcal{D}^{1}_T$ having residue $\mathcal{C}\in H^0(T[c])$ along the residue map
	\[
	H^{1}(T-T[c],\mathbb{R}) \to H^0(T[c])
	\]
	is equivalent to $d\omega = \delta_\mathcal{C}$ (where this means the linear combination of the currents of integration along points in the support of $\mathcal{C}$ corresponding to their coefficients). See for example \cite[(3.3)]{X} from the author's thesis: in general, \cite[\S3.2.2]{X} contains more details on the distributional de Rham complex along with proofs (or references to original proofs).
	
	Since the distributional de Rham complex is a functorial complex computing ordinary real cohomology, the equivariant cohomology of $T$ can thus be computed by the double complex $C^\bullet(\Gamma, \mathcal{D}^\bullet_T)$, where $\Gamma\subset \SL_2(\Z)$ acts on $T$-currents by pushforward. Analogously to the non-equivariant case, then, an element $\omega$ of this double complex restricts to a representative of a class in $H^{1}_\Gamma(T-T[c])$ with residue 
	\[
	[T[c]-c^2\{0\}]\in H^0(T[c])^\Gamma
	\]
	if and only if the total differential of $\omega$ is $\delta_{T[c]}-c^2\delta_0 \in C^0(\Gamma, \mathcal{D}^2_T)$. In particular, if we can find such a class which is invariant by $[a]_*$ for all integers $a$ relatively prime to $c$, then it will represent the class
	\[
	z_\Gamma^{(c)} \in H^{1}_\Gamma(T-T[c], \mathbb{R}).
	\]

	\subsection{Bernoulli polynomials}
	
	We now know that in principle, one can compute Eisenstein classes by finding suitable lifts inside the distributional de Rham complex.
	
	These elements will be built out of the periodic \emph{Bernoulli polynomials} of weight $1$
	\[
	B_1(z):= \{z\}-\frac{1}{2}
	\]
	where $\{z\}$ is the fractional part of $z$, i.e. its unique representative in $[0,1)$ modulo $1$, and weight-$2$
	\[
	B_2(z)=\{z\}^2 - \{z\}+\frac{1}{6}.
	\]
	As periodic functions, these Bernoulli polynomials can be viewed as defined on $\R/\Z$. Clearly, $B_1$ is not smooth, or even continuous; its graph is in the shape of a sawtooth. The weight-two $B_2$, meanwhile, is continuous, since $0^2-0+1/6=1^2-1+1/6$, and smooth away from the identity in $S^1$. Also away from the identity, it is easy to see that $B_2'(z)=2B_1(z)$. Furthermore, these polynomials satisfy the distribution relations 
	\[ B_1(z) = [a]_*B_1(z) := \sum_{z'\in [a]^{-1}z} B_1(z) \]
	\[
	B_2(z) = a [a]_*B_2(z) := a \sum_{z'\in [a]^{-1}z} B_2(z)
	\]
	for all $a\in \mathbb{N}$, as can be easily verified by hand.
	
	Despite not being smooth as functions, $B_1(z)$ and $B_2(z)$ can be considered as smooth $0$-currents on all of $S^1$, in the sense that
	\[
	\alpha \mapsto \int_{S^1} B_i(z) \alpha
	\]
	is a well-defined functional on smooth $1$-forms $\alpha$ on $S^1$, for $i=1,2$. Considered currents, we have that
	\[
	dB_1 = dz - \delta_0,
	\]
	\[
	dB_2 = 2 B_1\, dz
	\]
	the former of which can be verified by computing their Fourier coefficients.
	
	We now construct a certain element $\theta$ of total degree $1$ in $C^\bullet(\Gamma, \mathcal{D}_T^\bullet)$, i.e. in 
	\[
	C^0(\Gamma, \mathcal{D}_T^1) \oplus C^1(\Gamma, \mathcal{D}_T^0)
	\]
	such that its total differential is
	\[
	\delta_0 - dz_1\wedge dz_2 \in C^0(\Gamma, \mathcal{D}_T^2).
	\]
	In particular, we set
	\[
	C^0(\Gamma, \mathcal{D}_T^1) \cong \mathcal{D}_T^1 \ni \theta^{0,1} := B_1(z_1)\, \delta_{z_2=0} - B_1(z_2)\, dz_1
	\]
	Then we can compute that the derivative of this element as a current is 
	\[
	\delta_0 - dz_1\wedge dz_2
	\]
	as desired, so to define $\theta^{1,0}\in C^1(\Gamma, \mathcal{D}_T^0)$, it remains to find a lift of
	\[
	\partial_\Gamma \theta^{0,1} = \left( \gamma\mapsto (\gamma-1) [ B_1(z_1)\, \delta_{z_2=0} - B_1(z_2)\, dz_1 ]\right).
	\]
	Before discussing how to find these lifts, we observe the relation of any such lift to the Eisenstein class. The key observation is that the class $\theta^{0,1}$
	is trace-fixed, i.e. $[a]_*\theta^{0,1}=\theta^{0,1}$ for all $a\in \N$. 
	Then if we can find a lift $\theta^{1,0}$ of $\partial_\Gamma \theta^{0,1}$ which is similarly trace-fixed, we can conclude: 
	\begin{prop} \label{prop:comp}
		Given a such a trace-fixed element $\theta = \theta^{1,0}+\theta^{0,1}$ with $d_{tot}\theta = \delta_0 - dz_1\wedge dz_2$ as above, we have for any integer $c>1$ that
		\[
		[([c]^*-c^2)\theta] = z_{\Gamma}^{(c)} \in H^1_\Gamma(T-T[c]).
		\]
	\end{prop}
	\begin{proof}
		From the preceding discussion, we see that $([c]^*-c^2)\theta$ is an element with total differential $\delta_{T[c]}-c^2\delta_0$, which is fixed by $[a]_*$ for all $(a,c)=1$. We hence conclude its restriction to $T-T[c]$ is a closed (under $d_{tot}$) element which satisfies the properties characterizing $z_{\Gamma}^{(c)}$, and hence represents it in cohomology.
	\end{proof}
	
	\subsection{Finding lifts}
	In this section, we will find a definition of $\theta^{1,0}$ by finding suitable lifts of $\gamma \partial_\Gamma \theta^{0,1}$, and see how this leads directly to formulas for the periods of classical Eisenstein series.
	
	We begin by noting that
	\begin{equation} \label{eq:recur}
		(\gamma_2\gamma_1-1)\theta^{0,1} = \gamma_2(\gamma_1-1)\theta^{0,1} + (\gamma_2-1) \theta^{0,1}
	\end{equation}
	so the problem of finding a lift for $\Gamma$ can be reduced to finding preimages of 
	\[
	(\gamma-1)(B_1(z_1)\, \delta_{z_2=0} - B_1(z_2)\, dz_1)
	\]
	for $\gamma$ running over a set generators of $\Gamma$.
	
	\subsection{Lifts for $\SL_2(\Z)$ and classical Dedekind-Rademacher formulas} \label{section:sl2z}
	
	We therefore consider the case of $\Gamma = \SL_2(\Z)$, which is famously generated by the two matrices
	\[
	S = \begin{pmatrix}0& -1 \\ 1 & 0\end{pmatrix}, T= \begin{pmatrix} 1 & 1 \\ 0 & 1 \end{pmatrix}.
	\]
	For the matrix $S$, we find that 
	\begin{align}
		(S_*-1)( B_1(z_1)\, \delta_{z_2=0} - B_1(z_2)\, dz_1 )  & = B_1(z_2)\, \delta_{z_1=0} - B_1(z_1)\, \delta_{z_2=0} + B_1(z_1)\,dz_2 + B_1(z_2)\, dz_1 \\ & = d(-B_1(z_1)B_1(z_2))
	\end{align}
	and hence can take $-B_1(z_1)B_1(z_2)$ as our preimage.
	
	For the matrix $T$, we find that 
	\begin{align}
		(T_*-1)(B_1(z_1)\, \delta_{z_2=0} - B_1(z_2)\, dz_1)  & = B_1(z_2)d(-z_1+z_2) + B_1(z_2)dz_1 \\
		& = B_1(z_2)dz_2 \\
		& = d(B_2(z_1)/2).
	\end{align}
	Analogously, for $T^b$, for any $b\in \Z$, we get that $bB_2(z_1)/2$ is a lift of
	\[
	(T^b_*-1)(B_1(z_1)\, \delta_{z_2=0} - B_1(z_2)\, dz_1 ). 
	\]
	In principle, then, by expressing $\gamma \in \SL_2(\Z)$ in the form 
	\[
	ST^{b_1}ST^{b_2}\ldots ST^{b_k}
	\]
	for some finite string of nonnegative integers $b_1,\ldots,b_k$ (positive except for possibly $b_k$), we can recursively find the value of $\theta^{0,1}$ for $\gamma$ using \eqref{eq:recur}, and therefore compute
	\[
	\int_{\tau_0}^{\gamma \tau_0} E_{2}^{\alpha,\beta}(\tau)\, d\tau 
	\]
	if $\gamma$ fixes $(\alpha,\beta)\in \Z/N$ for some integer $N>1$. However, this expression would become quite involved due to its recursive definition. The intricacies can be packaged into a formalism of continued fractions, as in \cite{KM}, which we do not reiterate here. 
	
	Instead, we will follow a different approach to the definition of $\theta^{1,0}$ not requiring arbitrary-length decompositions of matrices.
	
	\subsection{Lifts for $\SL_2(\Q)$} \label{section:sl2q}
	
	To obtain the classical formulas \eqref{eq:eis}, we actually extend the range of our cocycle. The key observation is that all the expressions involved, from the original class 
	\[
	\delta_0 - dz_1\wedge dz_2,
	\]
	to the lifted class
	\[
	B_1(z_1)\, \delta_{z_2=0} - B_1(z_2)\, dz_1
	\]
	to the explicit lifts for $S$ and $T^b$
	\[
	- B_1(z_1)B_1(z_2),\,  bB_2(z_2)/2,
	\]
	are all invariant under the scalar pushforwards $[a]_*$ for any $a\in \N \setminus \{0\}$. For any module $M$ equipped with such a monoid action of $\N^\times$, we write $M^{(0)}$ for this \emph{trace-fixed} part; in particular, we can apply this to the pushforward action on the distributional de Rham complex of $T$. For this complex, we can observe:
	\begin{prop}
		The trace-fixed complex
		\[
		0\to (\mathcal{D}^{0}_T)^{(0)} \to (\mathcal{D}^{1}_T)^{(0)} \to (\mathcal{D}^{2}_T)^{(0)}
		\]
		is left-exact.
	\end{prop}
	\begin{proof}
		This can be proven identically to \cite[Lemma 6.2.1]{SV}, as the only trace-fixed cohomology of $T$ is in top degree.\footnote{In fact, one can show using Fourier series that the complex is almost right-exact as well, except for one dimension of cohomology on the right, but this is not necessary for our purposes. For details, see \cite{RX}, which generalizes much of this background to general-dimension tori.}
	\end{proof}
	Then the trace-fixed complex actually takes an action of $\GL_2^+(\Q)$ and not just $\SL_2(Z)$, since if we write for any $M\in \GL_2^+(\Q)$
	\[
	M = [a]^{-1} M'
	\]
	where $M'$ is an integral matrix, it is well-defined to give the action of $M$ on $(\mathcal{D}_T^i)^{(0)}$ as simply coinciding with that of $(M')_*$, since all scalars act trivially. Correspondingly, we can set $\Gamma = \GL_2^+(\Q)$, and ask for $\theta^{1,0}$ to lift 
	\[
	(\gamma - 1)(B_1(z_1)\, \delta_{z_2=0} - B_1(z_2)\, dz_1 )
	\]
	for any rational invertible orientation-preserving matrix $\gamma \in \GL_2^+(\Q)$.
	
	The upshot of making our group larger is that we have the \emph{Bruhat decomposition}
	\[
	\GL_2^+(\Q) = B\begin{pmatrix}0 & -1 \\ 1 & 0\end{pmatrix} B
	\]
	where $B\le \GL_2^+(\Q)$ is the upper triangular Borel subgroup. This Borel subgroup has the further decomposition
	\[
	B = U \rtimes D
	\]
	where 
	\[
	U = \begin{pmatrix} 1 & \bullet \\ 0 & 1 \end{pmatrix}, D= \begin{pmatrix} \bullet & 0 \\ 0 & \bullet \end{pmatrix}
	\]
	are respectively the upper unipotent and diagonal torus subgroups of $\GL_2^+(\Q)$. By taking advantage of these decompositions, then, we can turn the lifting problem for a general element of $\GL_2^+(\Q)$ into a problem for these three simple types of matrices: for the matrix
	\[
	S= \begin{pmatrix}0 & -1 \\ 1 & 0\end{pmatrix}
	\]
	we have already done this. For 
	\[
	\upsilon_b := \begin{pmatrix}1 & b \\ 0 & 1\end{pmatrix}
	\]
	for any $b\in \Q$, we can compute similarly to the case of $b\in \Z$ (in which case $\upsilon_b = T^b$) that a lift is given by
	\[
	bB_2(z_2)/2.
	\]
	Finally, any diagonal matrix in $D$ fixes $B_1(z_1)\, \delta_{z_2=0} - B_1(z_2)\, dz_1 $, so the corresponding lift is zero.
	
	This, as in the previous section, makes it in \emph{in principle} possible to write down the value of $\theta^{1,0}$ for any element of $\GL_2^+(\Q)$, but to obtain the classical formulas \eqref{eq:eis} we need an \emph{explicit} version of the Bruhat decomposition:
	For any 
	\[
	\gamma = \begin{pmatrix}
		a& b \\ 0 & d
	\end{pmatrix}\in B
	\]
	we can write
	\[
	\begin{pmatrix}
		a& b \\ 0 & d
	\end{pmatrix} = \begin{pmatrix}
		1& b/d \\ 0 & 1
	\end{pmatrix}\begin{pmatrix}
		a& 0 \\ 0 & d
	\end{pmatrix}.
	\]
	In the other Schubert cell, for 
	\[
	\gamma = \begin{pmatrix}
		a& b \\ c & d
	\end{pmatrix}\in \GL_2^+(\Q) - B,
	\]
	with $c\ne 0$, we can write
	\[
	\begin{pmatrix}
		a& b \\ c & d
	\end{pmatrix} = \begin{pmatrix}
		1 & a\cdot \text{sgn}(c) \\ 0 & |c|
	\end{pmatrix} \begin{pmatrix}
		0& -1 \\ 1 & 0
	\end{pmatrix} \begin{pmatrix}
		\text{sgn}(c)& 0 \\ 0 & \text{sgn}(c)
	\end{pmatrix}\begin{pmatrix}
		1& d/c \\ 0 & \frac{\det \gamma }{|c|}
	\end{pmatrix}
	\]
	
	Let us see how the formulas \eqref{eq:eis} now follow from these explicit decompositions. For $\gamma\in B$, we find that
	\[
	\theta^{1,0}(\gamma)  = \theta^{1,0}\begin{pmatrix}    1& b/d \\ 0 & 1
	\end{pmatrix} + \begin{pmatrix}
		1& b/d \\ 0 & 1
	\end{pmatrix}_* \theta^{1,0}\begin{pmatrix}
		a& 0 \\ 0 & d
	\end{pmatrix} = \theta^{1,0}\begin{pmatrix}    1& b/d \\ 0 & 1
	\end{pmatrix}  = -\frac{b}{2d} B_2(z_2)
	\]
	
	For $\gamma \in \GL_2^+(\Q) - B$, we use the further decompositions
	\[
	\begin{pmatrix}
		1 & a\cdot \text{sgn}(c) \\ 0 & |c|
	\end{pmatrix} =  \begin{pmatrix}
		1& a/c \\ 0 & 1 
	\end{pmatrix} \begin{pmatrix}
		1& 0 \\ 0 & |c|
	\end{pmatrix} 
	\]
	and
	\[
	\begin{pmatrix}
		1& d/c \\ 0 & \frac{\det \gamma}{|c|}
	\end{pmatrix} = \begin{pmatrix}
		1& \frac{d\cdot \text{sgn}(c)}{\det \gamma} \\ 0 & 1 
	\end{pmatrix} \begin{pmatrix}
		1& 0 \\ 0 & \frac{\det \gamma}{|c|}
	\end{pmatrix}
	\]
	coming from $B = U \rtimes D$.
	
	Recalling that the positive diagonal torus $\Q_{>0}^\times \times \Q_{>0}^\times\subset \GL_2(\Q)$ acts trivially on $B_1(z_1)\, \delta_{z_2=0} - B_1(z_2)\, dz_1$, using the recursive principle \eqref{eq:recur} we find using our explicit Bruhat decomposition that the lift $\theta^{1,0}(\gamma)$ can be written as the sum of three terms:
	\begin{enumerate}
		\item The lift 
		\[
		\theta^{1,0}(\upsilon_{a/c})= \frac{a}{2c} B_2(z_2).
		\]
		\item The lift 
		\[
		\begin{pmatrix}
			1 & a\cdot \text{sgn}(c) \\ 0 & |c|
		\end{pmatrix}_* \theta^{1,0}(S) = - \begin{pmatrix}
			1 & a\cdot \text{sgn}(c) \\ 0 & |c|
		\end{pmatrix}_* B_1(z_1)B_1(z_2).
		\]
		\item The lift 
		\[
		\begin{pmatrix}
			1 & a\cdot \text{sgn}(c) \\ 0 & |c|
		\end{pmatrix}_* S_* \theta^{1,0}(\upsilon_{d/\det \gamma}) = \frac{d}{2\det \gamma} \begin{pmatrix}
			a\cdot \text{sgn}(c) & -1 \\ |c| & 0 
		\end{pmatrix}_* B_2(z_2)
		\]
	\end{enumerate}
	
	\subsection{Comparison of pullbacks}
	
	Let us see how the lift $\theta^{1,0}$ of the previous section recovers the original formula \eqref{eq:eis} for matrices in $\SL_2(\Z)$. Our idea is that if we have a $\Gamma$-fixed torsion section $(\alpha, \beta): * \hookrightarrow T$, we would like to compare the evaluation $(\alpha,\beta)^*\theta^{1,0}$ with the period $E_{\alpha,\beta}^2$, by passing via the cohomology class $_{c}z_\Gamma$ which both are related to.
	
	The problem is that $\theta^{1,0}$ is valued in $0$-currents rather than functions, whose evaluation at points is not in general well-defined: currents cannot be pulled back by closed inclusions in general. Thus, we cannot even a priori evaluate our cocycle at $(\alpha,\beta)$, let alone compare the result with the analogous pullback of the Eisenstein cohomology class $_{c}z_\Gamma$. 
	
	The first technical tool we need to remedy this is the introduction of a variant of the distributional de Rham complex: For any $\Gamma\subset \GL_2^+(\Q)$, let $H_\Gamma\subset T$ be the $\Gamma$-orbit of the lines $\{z_1=i/c\}$ and $\{z_2=i/c\}$ inside $T$ as $i$ varies over $\{0,1,\ldots, i-1\}$; this is a union of infinitely many subtori translated by $c$-torsion points. For any finite subarrangement $H\subset H_\Gamma$, we define a complex $\mathcal{D}_{T,H}^\bullet$ via the pullback square 
	\begin{equation}
		\begin{tikzcd}
			\mathcal{D}_{T,H}^\bullet \arrow[r] \arrow[d] & \mathcal{D}_{T}^\bullet \arrow[d]\\
			\Omega_{T-H}^\bullet \arrow[r] & \mathcal{D}_{T-H}^\bullet
		\end{tikzcd}
	\end{equation}
	which results in an identification of $\mathcal{D}_{T,H}^i$ with the $i$-currents such that their restriction to $T-H$ are given by smooth $i$-forms. Here, the bottom horizontal map is the earlier-defined inclusion, and the right vertical map is the restriction dual to the pushforward of compactly-supported differential forms. We define then
	\[
	\mathcal{D}_{T,H_\Gamma}^\bullet := \varinjlim_H \mathcal{D}_{T,H}^\bullet 
	\]
	where the limit runs over $H$ finite subarrangements of $H_\Gamma$, along the natural inclusion maps. The group $\Gamma$ permutes the pullback diagrams for each $H$ (sending it to that of $\gamma H$), and these assemble to give a pushforward action on $\mathcal{D}_{T,H_\Gamma}^\bullet$. Further, because the bottom row in each pullback diagram is a quasi-isomorphism, we see that $\mathcal{D}_{T,H_\Gamma}^\bullet$ computes the cohomology of $T$, just as $\mathcal{D}_{T}^\bullet$ does. Furthermore, analogously to the full distributional de Rham complex, we have a left exact sequence
	\[
	0 \to (\mathcal{D}_{T,H_\Gamma}^0)^{(0)}\to (\mathcal{D}_{T,H_\Gamma}^1)^{(0)}\to (\mathcal{D}_{T,H_\Gamma}^2)^{(0)}
	\]
	The important new phenomenon for us is that if $\Gamma$ fixes a point $(\alpha,\beta)\in T$ not lying in $H_\Gamma$, then there is a composite pullback map
	\[
	\mathcal{D}_{T,H_\Gamma}^0\to \varinjlim_H \Omega^0_{T-H} \xrightarrow{(\alpha, \beta)^*} \R
	\]
	which induces the pullback $(\alpha,\beta)^*$ on the level of real cohomology. One consequence of this exactness is that the required property of $\theta^{1,0}$, i.e. that of lifting
	\[
	\gamma \mapsto (\gamma-1) (\theta^{0,1}),
	\]
	must characterize it uniquely up to coboundaries if we insist that it be valued in trace-fixed $0$-currents. 
	
	With these ingredients, we are now able to prove the central comparison result:
	\begin{thm}
		If $\Gamma\subset \GL_2(\Q)^+$ fixes $(\alpha,\beta)\in T[N]\setminus \{0\}$, then for all $\gamma \in \Gamma$,
		\[
		\frac{1}{2} \Phi_{\alpha, \beta}(\gamma) = (\alpha,\beta)^*\theta^{1,0}(\gamma)
		\]
		where here, we somewhat abusively consider $\theta^{1,0}$ to be valued in the \emph{functions} given by the formulas in section \ref{section:sl2q} to make sense of the pullback in general.\footnote{This theorem still holds if we define the lift $\theta^{1,0}$ according to the formulas in section \ref{section:sl2z} instead, with the same proof. Since our main focus is the lifts taking advantage of the Bruhat decomposition, we choose to phrase everything in terms of the section \ref{section:sl2q} formulas.}
	\end{thm}
	\begin{proof}
		We begin by treating only the case where either $\alpha\ne 0$ or $\beta\ne 0$, as this is easier. In this case, we see that in fact according to our formulas,
		\[
		([c]^*-c^2)\theta^{1,0}(\gamma) \in (\mathcal{D}^0_{T,H_\Gamma})^{(0)}
		\]
		for all $\gamma\in \Gamma$, since its derivative is of the form
		\[
		(\gamma-1) (B_1(z_1)\, \delta_{z_2=0} - B_1(z_2)\, dz_1)
		\]
		which corresponds to a smooth $1$-form once restricted away from $H_\Gamma$. Then since $(\alpha,\beta)\not\in H_\Gamma$, we conclude that it makes sense to consider the pullback $(\alpha,\beta)^*([c]^*-c^2)\theta$; since this induces the pullback on cohomology, we have
		\[
		\frac{1}{2}\Phi_{\alpha,\beta}(\gamma)=(\alpha,\beta)^*\frac{1}{1-c^2} z_{\Gamma}^{(c)} = \frac{1}{1-c^2}(\alpha,\beta)^*([c]^*-c^2)\theta(\gamma)=(\alpha,\beta)^*\theta(\gamma)
		\]
		for any $c$ such that $(c\alpha,c\beta)=(\alpha,\beta)$, i.e. $c\equiv 1 \pmod{N}$. For these $(\alpha,\beta)$, the result follows, since the only non-vanishing component of $\theta$ after pullback is $\theta^{1,0}$.
		
		In the other case when either $\alpha=0$ or $\beta=0$, the problem is that though the formulas in section \ref{section:sl2q} when treated as \emph{functions} make sense to evaluate at $(\alpha,\beta)$, since they are not in the locus of smoothness, this evaluation has no \emph{a priori} cohomological meaning: heuristically, the currents ``do not know'' about the values of functions like $B_1(z)$ at the boundary points where they are discontinuous. To remedy this, we have the following lemma:
		
		\begin{lem} \label{lem:inj}
			There is an injection
			\[
			M\hookrightarrow (\mathcal{D}^0_{T,H_\Gamma})^{(0)}
			\]
			where $M$ is defined to be the module of functions on $T-\{0\}$ spanned by the $\GL_2^+(\Q)$-orbit (under pushforward) of $B_2(z_2)$ and $B_1(z_1)B_1(z_2)$, given by sending
			\[
			f\mapsto \left(\eta\mapsto \int_T f \eta\right).
			\]
		\end{lem}
		\begin{proof}[Proof of lemma]
			The only property of this inclusion whose proof is not fully analogous to the proof for the inclusion of smooth $0$-forms is the injectivity: whereas for smooth forms it follows by using a partition of unity to work locally, this is far from true in general if we weaken the continuity assumptions.\footnote{For example, it fails even if we replace $T-\{0\}$ by $T$ in this very lemma, thanks to the identity
				\[
				2(B_1(z_1)B_1(z_2)+B_1(z_2)(B_1(-z_1-z_2))+B_1(-z_1-z_2)B_1(z_1))+B_2(z_1)+B_2(z_2)+B_2(-z_2-z_2)=0
				\]
				which holds identically away from the zero section (where the left-hand side evaluates to $1/2$). This nonzero linear combination of Bernoulli polynomials in $M$ therefore gives rise to the zero current.}
			
			Noting that the loci of discontinuity of functions in $M$ can only be along codimension-$1$ subtori embedded as subgroups, coming from terms of the form $B_1(z_1)B_1(z_2)$ and its translates. We note the following property of this function: if $x=(x_1,x_2)\in T$ is a point lying on one of the subtori of discontinuity $S\subset T$  (so $x_1=0$ or $x_2=0$) but not equal to zero, then $S$ divides any sufficiently small neighborhood around $z$ into two connected components, which we label $+$ and $-$ arbitrarily. Then we observe that the limits from these two components
			\[
			L_+=\lim_{z\to x^+} B_1(z_1)B_1(z_2),L_-=\lim_{z\to x^-} B_1(z_1)B_1(z_2)
			\]
			exist, and furthermore that they average to zero. By moving this argument around by the general linear action, this applies to \emph{any} nonzero point on a codimension $1$ discontinuity stratum of a function in the orbit of $B_1(z_1)B_1(z_2)$.
			
			Now consider an arbitrary $f\in M$, and suppose that $f$ gives the trivial $0$-current when considered as a kernel of integration. Then, $f$ must be identically zero outside a finite union of subtori through the identity. Consider an arbitrary nonzero point $x$ on one of these subtori $S\subset T$, and pick some decomposition
			\[
			f= f_1+f_2
			\]
			where $f_1$ consists of terms coming from the orbit of $B_2$ or $B_1(z_1)B_1(z_2)$ which \emph{do not} have a discontinuity along $S$, and $f_2$ consists of terms from the orbit of $B_1(z_1)B_1(z_2)$ which do have a discontinuity along $S$. Then we find that
			\begin{align}
				\frac{1}{2}\left(\lim_{z\to x^+} f_2(z) + \lim_{z\to x^-} f_2(z)\right) &= \frac{1}{2}\left(\lim_{z\to x^+} f(z)-f_1(z) + \lim_{z\to x^-} f(z)-f_1(z)\right) \\
				&= \frac{1}{2}\left(\lim_{z\to x^+} -f_1(z) + \lim_{z\to x^-} -f_1(z)\right)
			\end{align}
			since $f$ is identically zero on a neighborhood of $x$ in $T-S$. But $f_1$ is continuous in a neighborhood of $x$ by assumption, so this expression is just $-f_1(x)$. On the other hand, the average of the two limits we started with is zero from the preceding discussion, so we conclude that $f_1(x)=0$. We also have $f_2(x)=0$ because functions in the orbit of $B_1(z_1)B_1(z_2)$ are zero along their discontinuity loci by construction. We hence conclude that $f(x)=0$; since this applies to any nonzero point $x$, we conclude that $f$ is the zero function on $T-\{0\}$. This concludes the proof of injectivity.
		\end{proof}
		
		\begin{rem}
			The fact that this injection only holds on $T-\{0\}$, and not $T$, also has the following significance: if we had injectivity also at zero, then (by the argument below) \eqref{eq:classical} would yield a valid cocycle even for $\alpha=\beta=0$. This ``$\Phi_{0,0}$'' in fact fails to be a cocycle: it is known in the literature as the \emph{Dedekind symbol} (e.g., in \cite{Duke}), whose coboundary gives a representative of the Euler class for $\SL_2(\Z)$. Write $M_0$ for the analogue of $M$ on all of $T$: then by y using the long exact sequence in $\SL_2(\Z)$-cohomology associated to the short exact sequence
			\[
			0\to \Z \to M_0\to M\to 0
			\]
			one can recover formulas for this Euler coboundary by explicit describing the image of $\theta$ under $H^1(\Gamma, M)\to H^2(\Gamma, \Z)$ (which is the obstruction to lifting $\theta$ to $H^1(\Gamma, M_0)$).
		\end{rem}
		
		To conclude the proof in general, let $T[N]'\subset T$ be the subset of primitive $N$-torsion sections, i.e. those of exact order $N$. Fix a point $x_0\in T[N]'$ with two nonzero coordinates, and let $\Gamma_1(x_0)\subset \SL_2(\Z)$ be the congruence subgroup fixing $x_0$. Then we have an identification
		\[
		\text{Ind}_{\Gamma_1(x_0)}^{\SL_2(\Z)} \Q \cong \hom_{\text{Sets}}(T[N]
		, \Q)
		\]
		since $\SL_2(\Z)$ permutes $T[N]'$ transitively and $\Gamma_1(x_0)$ is the stabilizer of $x_0$. We also have a $\SL_2(\Z)$-equivariant map
		\[
		j: M \to \hom_{\text{Sets}}(T[N]', \Q), f\mapsto (x\mapsto x^* f)
		\]
		which yields by pushforward a cocycle 
		\[
		j_*\theta\in C^1(\SL_2(\Z),\text{Ind}_{\Gamma_1(x_0)}^{\SL_2(\Z)} \Q)
		\]
		such that its restriction to $\Gamma_1(x_0)$ is the cocycle $x_0^*\theta$, which we know represents the class $\frac{1}{2}\Phi_{x_0}$. Thus by Shapiro's lemma, the class $[j_*\theta]$ coincides with the restricted Eisenstein class on $T'[N]$
		\[
		\frac{1}{1-c^2}z^{(c)}_{\SL_2(\Z)} \in H^1(\SL_2(\Z), H^0(T[N]')).
		\]
		This implies that for \emph{any} $(\alpha,\beta)\in T[N]'$, with stabilizer $\Gamma$, we have
		\[
		[(\alpha,\beta)^*\theta] = [(\alpha,\beta)^*j_*\theta] = (\alpha,\beta)^*\frac{1}{1-c^2}z^{(c)}_{\Gamma} = \frac{1}{2}[E_{\alpha,\beta}^2]\in H^1(\Gamma, \Q).
		\]
		Applying this argument starting with some $x_0$ with at least one nonzero coordinate (via the previous argument) completes our proof.
		
	\end{proof}
	
	To conclude the article, we compare our explicit specializations of $\theta^{1,0}$ to the classical formulas \eqref{eq:classical}, when $\gamma\in \SL_2(\Z)$: thanks to the preceding theorem, these values should coincide precisely for all torsion sections $(\alpha,\beta)$, up to a factor of $-1/2$.
	
	We recall our formulas for $\theta^{1,0}$ stemming from the Bruhat decomposition: For $\gamma\in B$, we found
	\[
	\theta^{1,0}(\gamma)  = \frac{b}{2d} B_2(z_2)
	\]
	whose pullback by $(\alpha, \beta)$
	\[
	\frac{b}{2d} B_2(\beta)
	\]
	clearly aligns with the classical formula for $\frac{1}{2}\Phi_{\alpha,\beta}$ on upper triangular matrices.
	
	For $\gamma \in \GL_2^+(\Q) - B$, we recall we had $\theta^{1,0}(\gamma)$ as the sum of three terms
	\[
	\frac{a}{2c} B_2(z_2)  - \begin{pmatrix}
		1 & a\cdot \text{sgn}(c) \\ 0 & |c|
	\end{pmatrix}_* B_1(z_1)B_1(z_2)  +\frac{d}{2\det \gamma} \begin{pmatrix}
		a\cdot \text{sgn}(c) & -1 \\ |c| & 0 
	\end{pmatrix}_* B_2(z_2).
	\]
	After pullback by $(\alpha,\beta)$, the first term becomes 
	\[
	\frac{a}{2c} B_2(\beta).
	\]
	Recalling that $\gamma$ is assumed to have integer entries, the second term becomes
	\begin{align}
		-(\alpha,\beta)^* \begin{pmatrix}
			1 & a\cdot \text{sgn}(c) \\ 0 & |c|
		\end{pmatrix}_* B_1(z_1)B_1(z_2)  & = 
		- \sum_{\substack{x_1+a\cdot \text{sgn}(c) x_2 = \alpha\\ |c| x_2=\beta }} B_1(x_1)B_1(x_2)
		\\
		&=  - \sum_{i=0}^{|c|-1} B_1\left(a\frac{\beta+i}{c}-\alpha \right) B_1\left(\frac{\beta+i}{|c|}\right).
	\end{align}
	
	Finally, recalling that $\det \gamma=1$, the third term becomes
	\begin{align}
		\frac{d}{2}  (\alpha,\beta)^* \begin{pmatrix}
			a \cdot \text{sgn}(c) & -1 \\ |c| & 0 
		\end{pmatrix}_* [\text{sgn}(c)]_*  B_2(z_2) & =  \frac{d}{2} \sum_{i=0}^{|c|-1} \text{sgn}(c)B_2\left( a \frac{\beta+i}{|c|} - \text{sgn}(c) \alpha\right) \\
		& = \frac{d}{2} \sum_{i=0}^{|c|-1} \text{sgn}(c)B_2\left(\frac{\beta+i}{|c|}\right) \\
		& = \frac{d}{2c} B_2(\beta)
	\end{align}
	by the distribution property of $B_2$ combined with the relation
	\[
	\frac{\beta}{c} = \frac{a}{c} \beta -\alpha
	\]
	coming from the fact that $\gamma^{-1}(\alpha,\beta)=(\alpha,\beta)$, i.e. $\gamma$ fixes the torsion point of evaluation.
	
	Combining these three terms, we obtain precisely the classical formula \eqref{eq:classical} for $\frac{1}{2}\Phi_{\alpha,\beta}$ on the other Schubert cell.
	
	\subsection{Future directions} 
	One extension of the methods of this article, the subject of work in progress, is to compute a formula for the $(\GL_2,\GL_2)$-Eisenstein theta lift of \cite{BCG}: in this setting, $(S^1)^2$ is replaced by the square of an elliptic curve, the Bernoulli polynomials by theta series $E_1(\tau,z),E_2(\tau,z)$ whose pullbacks are weight-$1$ and $2$ Eisenstein series, and the equivariant de Rham complex by the weight-$2$ equivariant Dolbeault complex. Otherwise, the method goes through in the same way, though the discontinuity phenomenon exhibited by $B_1$ is replaced by a logarithmic singularity of $E_1$. Unlike the present article, the formulas in the theta lift setting are novel. (In parallel, we also are writing a preprint \cite{X3} on the ``stabilized'' approach to this theta lift, though this in the setting of $K$-theory rather than the differential forms which are their regulators.)
	
	The more ``obvious'' direction of generalization is to replace $\GL_2$ with $\GL_n$, with its action on the $n$-torus, and to obtain formulas for the periods of Eisenstein series of this larger group in terms of sums of products of Bernoulli polynomials of total degree $n$: this would be the ``algebraic'' analogue of the Bernoulli formulas in \cite{Scz2}, just as the present article is to \cite{Scz1}. However, though most of our formalism goes through in this case largely untouched, the combinatorics of finding lifts becomes much more delicate as $n$ increases: instead of just taking the Bruhat decomposition, one needs to consider families of nested parabolics to lift the Bernoulli currents at various stages. It is presently unclear to us what a good systematic method for keeping track of these lifts might be; an approach inspired by the stratifications used in \cite{Scz2} may be useful. A topological approach, as we used in a slightly different setting \cite{X2} for example, would be ideal; however, it would have to be much more complicated than in loc. cit. to account for the phenomena arising from various parabolic subgroups.
	
	The other issue that occurs in this generalization is that the analogue of Lemma \ref{lem:inj} is only true in codimension $1$ generally, as suggested by the main theorem of \cite{Scz2}, so we would not know how to treat pullbacks torsion sections with multiple zero coordinates in a uniform one. (One could simply take a different initial lift to avoid the issue with any given fixed torsion section, but then this makes the results less systematic.) Because of these technical difficulties, this direction of generalization is currently not the subject of our work; however, it would be interesting to understand how they can be circumvented.
	
	Finally, it would be interesting to see if similar formulas could be obtained in this way for $\GL_2$-Eisenstein series of higher weight, by introducing twisted coefficients into our complexes. It is, however, not obvious to us what precisely the twisted analogue of our periodic Bernoulli polynomials should be.
	
	\printbibliography
	
\end{document}